\pdfoutput=1
\documentclass{amsart}
\usepackage{booktabs}
\usepackage{amssymb}
\usepackage{tikz-cd}
\makeindex

\usepackage{hyperref}
\hypersetup{
colorlinks=true,       
    linkcolor=blue,         
    citecolor=blue,        
    filecolor=blue,      
    urlcolor=blue         
    }
\usepackage[all]{xy}
\usepackage{tikz}
\usepackage{todonotes}
\usepackage[alphabetic,backrefs,lite]{amsrefs}
\usepackage{enumerate}
\usepackage{verbatim}
\usepackage{multirow}
\usepackage{pgf,tikz}
\usetikzlibrary{arrows}
\usetikzlibrary{shapes.geometric}

\usepackage[many]{tcolorbox}  
\definecolor{main}{HTML}{5989cf}
\definecolor{sub}{HTML}{cde4ff} 
\newtcolorbox{boxE}{
    enhanced, 
    boxrule = 0pt, 
    borderline = {0.75pt}{0pt}{main}, 
    borderline = {0.75pt}{2pt}{sub} 
}

\CompileMatrices
\newtheorem*{introthm}{Theorem}

\newtheorem{theorem}{Theorem}[section]
\newtheorem{lemma}[theorem]{Lemma}

\newtheorem{proposition}[theorem]{Proposition}
\newtheorem{corollary}[theorem]{Corollary}
\theoremstyle{definition}

\newtheorem{remark}[theorem]{Remark}

\def\cc{{\mathbb C}}

\def\pp{{\mathbb P}}

\def\Div{\operatorname{Div}}

\def\Cl{\operatorname{Cl}}
\def\Pic{\operatorname{Pic}}
\def\Ext{\operatorname{Ext}}

\def\BNef{\operatorname{BNef}}
\def\Bs{\operatorname{Bs}}

\renewcommand{\div}{\operatorname{div}}

\begin{document}

\title{Cox rings of nef anticanonical rational surfaces}

\author{Michela Artebani}
\address{
Departamento de Matem\'atica, \newline
Universidad de Concepci\'on, \newline
Casilla 160-C,
Concepci\'on, Chile}
\email{martebani@udec.cl}

\author{Sofía Pérez Garbayo}
\address{
Departamento de Matem\'atica, \newline
Universidad de Concepci\'on, \newline
Casilla 160-C,
Concepci\'on, Chile}
\email{sperez2017@udec.cl}

\thanks{The authors have been partially 
supported by Proyecto FONDECYT Regular 
N. 1211708}

\begin{abstract}
This paper deals with the problem of computing a generating set for the Cox ring $R(X)$
of a smooth projective rational  surface $X$ with nef anticanonical class. 
In case $R(X)$ is finitely generated, we show that the degrees of its generators 
are either classes of negative curves, elements of the Hilbert basis of the nef cone 
or certain ample classes of anticanonical degree one, which only appear when $X$ is a rational elliptic surface of Halphen index $m>2$. Moreover, we partially characterize which elements of the Hilbert basis of the nef cone are irredundant for generating $R(X)$.
We apply this result to compute explicit minimal generating sets for Cox rings of some rational 
elliptic surfaces of Halphen index $>1$.
\end{abstract}
\maketitle

\section*{Introduction}
Let $X$ be a smooth projective surface over the complex numbers with finitely generated and torsion free divisor class group $\Cl(X)$.
The Cox ring of $X$ is the $\cc$-algebra graded by $\Cl(X)$ defined as follows (see \cite{ADHL}):
\[
R(X):=\bigoplus _{[D]\in \Cl(X)}\Gamma(X,\mathcal O_X(D)).
\]
In case $R(X)$ is finitely generated the surface $X$ can be obtained 
as a quotient $p_X:\hat X\to X$ of a big open subset $\hat X$ of the affine variety associated to $R(X)$ 
by the action of a torus. 
The map $p_X$ is a universal torsor over $X$, an object which has been introduced to study rational points of surfaces defined over number fields \cite{ADHL}*{\S 6.1.2}.
In the case of smooth surfaces with finitely generated Cox ring, 
the graded algebra $R(X)$ determines $X$ up to isomorphism.
For these reasons the Cox ring is a fundamental algebraic invariant associated to $X$.
 
The classification of surfaces with finitely generated Cox ring is an open problem.
We know that this property is related to the positivity of the anticanonical class of $X$, for example 
$R(X)$ is finitely generated when $-K_X$ is big \cite{TVA} or when a positive multiple of $-K_X$ defines 
an elliptic fibration with a finite number of $(-1)$-curves \cite{ArLa}.
In \cite{BP} the authors proved that the Cox ring of a del Pezzo surface with $\rho(X)\geq 3$ is generated by defining sections of $(-1)$-curves 
together with a basis of anticanonical sections when $\rho(X)=9$. In \cites{DD,DDD} Derenthal provided an inductive method to compute Cox rings 
of weak del Pezzo surfaces and computed the Cox ring when the degree is $\geq 3$ and the Cox ring has a unique relation.

In \cite{AGL} it is proved that the Cox ring of a jacobian extremal rational elliptic surface is generated by defining sections of smooth rational curves, conic bundles, twisted cubics and anticanonical sections.

The purpose of this paper is to generalize both results providing a general approach 
for the computation of a generating set of $R(X)$ for smooth projective rational surfaces with nef anticanonical class. This approach is based on the idea that the Hilbert basis 
of the nef cone has a key role. 
This leads to the following result (see Theorem \ref{main1} and Theorem \ref{main2}).

\begin{introthm}
Let $X$ be a smooth projective rational surface such that $-K_X$ is nef and whose Cox ring $R(X)$ is finitely generated. 
The degrees of a minimal homogeneous generating set of $R(X)$ belong to the following list:
\begin{enumerate}
\item classes of curves with negative self-intersection,
\item  elements of the Hilbert basis of the nef cone,
\item ample classes of the form $-\alpha K_X+E$, where $2\leq \alpha< m$, $E$ is the class of a
$(-1)$-curve  and $X$ is an elliptic surface of index $m>1$.
\end{enumerate}
\end{introthm}

Moreover, we partially characterize which elements  of the Hilbert basis of the nef cone are irredundant (see Theorem \ref{necesarios}).

We now give a short description of the content of the paper.

Section 1 contains preliminary results on rational surfaces with nef anticanonical class 
and on Cox rings of surfaces. In the first subsection we recall a result by Harbourne 
about base loci and cohomology of nef divisors on anticanonical surfaces. Moreover we 
prove a result which describes the Hilbert basis of the nef cone of such surfaces.
In the subsequent section we introduce Cox rings, we characterize nef anticanonical surfaces 
with finitely generated Cox ring and we recall some techniques based on Koszul type sequences which 
allow to prove that certain classes of divisors can not appear as degrees of a minimal set of homogeneous 
generators of the ring.

Section 2 contains the proof of the main theorem. This is achieved by proving several auxiliary results, based on Koszul type exact sequences, which allow to show that nef classes which are sums of at least two elements in the Hilbert basis of the nef cone are not necessary to generate the Cox ring, with some very special exceptions. 
At the end of the section we partially characterizes which elements in the Hilbert basis of the nef cone are necessary to generate the Cox ring. 

Section 3 contains examples, mostly for Cox rings of 2-Halphen elliptic surfaces, but also for some higher index ones, as well for some weak del Pezzo surfaces.

  \vspace{4mm}

\noindent {\em Acknowledgments.} We thank Antonio Laface for several useful discussions and for sharing with us Magma programs which allow us to compute the classes of negative curves of many examples of rational elliptic surfaces.

\section{Preliminaries}

\subsection{Anticanonical surfaces}

Let $X$ be a smooth projective rational surface $X$ whose anticanonical class is nef.
By the Riemann-Roch formula and Serre duality the anticanonical class is also effective, thus these are {\em anticanonical surfaces}  \cite{Har}.
Observe that by the adjunction formula the 
integral curves of negative self-intersection of $X$ are either $(-1)$-curves or $(-2)$-curves.

If $-K_X$ is big, then $X$ is a {\em weak del Pezzo surface}, i.e. it is isomorphic to either $\pp^1\times \pp^1$, the Hirzebruch surface $\mathbb F_2$ 
or a blow up of $\pp^2$ in $r \leq 8$ points, where at each step no blown up point lies on a $(-2)$-curve.

Otherwise $K_X^2=0$ and either $h^0(-mK_X)=2$ for some $m>0$ or  $h^0(-mK_X)=1$ for all $m\geq 1$.
In the first case we will say that $X$ is a {\em rational elliptic surface of index} $m$ and can be obtained blowing up 
the base points of an Halphen pencil, i.e. a pencil of plane curves of degree $3m$ with nine (possibly infinitely near) base points of multiplicity $m$. 
This pencil induces an elliptic fibration on $X$ which, in case $m>1$, has a unique fiber of multiplicity $m$ given by $mF$, where $F$ is the unique 
curve in $|-K_X|$, and $F_{|F}$ is a non-trivial $m$-torsion element in $\Pic^0(F)$ \cite{HL}.

We start recalling some properties of nef divisors on nef anticanonical surfaces.
Observe that any nef divisor $D$ is effective and $h^2(X,D)=0$ \cite{Har}*{Corollary II.3}.

\begin{proposition}\label{harbourne}
Let $X$ be a smooth projective rational surface with $-K_X$ nef and $D$ be a non-zero nef divisor of $X$.
\begin{enumerate}
\item If $X$ is a weak del Pezzo surface, then $h^1(X,D)=0$. 
Moreover $|D|$ is base point free unless $\rho(X)=9$ and $D\sim -K_X$ (in this case the base locus consists of one point).
\item If $X$ is a rational elliptic surface of index $m=1$, then $h^1(X,D)=0$ if and only if $-K_X\cdot D>0$.  
Moreover  $|D|$ is base point free unless $D\sim -aK_X+E$, where $a>0$ is an integer and $E$ is a $(-1)$-curve 
(in this case the base locus is $E$).
\item If $X$ is a rational elliptic surface of index $m>1$, then  $h^1(X,D)=0$ if and only if either $-K_X\cdot D>0$ 
or $D\sim -aK_X$ for some integer $a\leq m-1$.
Moreover $|D|$ is base point free unless either $D \sim -aK_X+E$ where $a>0$ is an integer, $E$ 
is a $(-1)$-curve (in this case the base locus is a point in $F\in |-K_X|$),
or $D\sim -aK_X$ for some integer $a>0$ non-divisible by $m$ (and the base locus is $F\in |-K_X|$).
\end{enumerate}
\end{proposition}

\begin{proof} The proof follows from \cite{Har}*{Theorem III.1}. We will only explain how to deduce (i) and (iii) from such result.
 
In case (i) observe that $-K_X\cdot D>0$ and $K_X^\perp$ is negative definite by the Hodge index theorem. 
It follows from the previous theorem that the only case when $|D|$ has base locus 
is when $-K_X\cdot D=1$ and $|-K_X|$ has base locus, i.e. when $K_X^2=1$, and the base locus of $|D|$ is the base point of $|-K_X|$.
In this case since $-K_X\cdot (D+K_X)=0$, $K_X^2>0$ and $(D+K_X)^2\geq 0$, then by the Hodge index theorem $D+K_X\sim 0$ 
(see also \cite{ADHL}*{Proposition 5.2.2.4}).

In case (iii)  $|D|$ can have base locus only if either $-K_X\cdot D=0$ or $-K_X\cdot D=1$.
In the first case $D\sim -aK_X$ 
for some integer $a>0$ non-divisible by $m$ by the Hodge index theorem.
In the second case observe that an integral curve $C$ in $K_X^\perp$ with $h^1(X,C)=1$ is linearly equivalent to $-mK_X$.
Since $m>1$ there are no smooth rational curves $N$ of negative self-intersection such that $N\cdot C\leq 1$. 
Thus the previous theorem implies that the base locus of $|D|$ contains no curves and it consists of at most one 
point in $F\in |-K_X|$. 

Finally, we prove that in case (iii), if $D$ is nef and $-K_X\cdot D=1$, then $D \sim -aK_X+E$ where $a>0$ is an integer and $E$ 
is a $(-1)$-curve. 
Observe that by the genus formula $D^2$ is odd, say $D^2=2s+1$ for some positive integer $s$. 
The divisor $R:=D+(s+1)K_X$  satisfies $R^2=-1$ and $-K_X\cdot R= 1$, thus it is effective by the Riemann Roch formula and Serre duality. 

Decompose $R$ as a sum of integral curves:
$$R=  E + G_1+\dots+G_n,$$
where $-K_X\cdot E=1$ and $-K_X\cdot G_i=0$ for all $i$.
Observe that the curves $G_i$ are either $(-2)$-curves or linearly equivalent to a multiple of $-K_X$, 
since they are contained in a fiber of the elliptic fibration.
Since $R$ is not nef (it has negative square), there exists a negative curve $C$ with $R\cdot C<0$. 
But $R\cdot C\geq 0$ if $C$ is a $(-2)$-curve, thus $C=E$ must be a $(-1)$-curve. 
Moreover
$$0>R\cdot E=-1 + E\cdot (G_1 + \dots+ G_n)$$
and $E\cdot G_i\geq 0$, thus  $E\cdot G_i = 0$ for all $i$. Finally,
        $$-1=R^2= (G_1 + \dots + G_n)^2 -1 $$
        and then $(G_1 + \dots + G_n)^2=0$, which implies \cite{barth peters}*{Lemma 8.2, III} that 
        $G_1+\dots+G_n$ is linearly equivalent to a multiple of $-K_X$. 
Thus $D\sim -a K_X+E$. 
\end{proof}

We now characterize the elements of the Hilbert basis of the nef cone of $X$, which will be denoted by $\BNef(X)$.

 \begin{proposition} \label{bajar} Let $X$ be a smooth projective rational surface with $-K_X$ nef, $\rho(X)\geq 3$
and with polyhedral effective cone. The Hilbert basis of the nef cone of $X$ consists of classes of the following types:
\begin{itemize}
     \item
     $\pi^*(H)$, where $\pi:X\to \mathbb{P}^1$ is a birational morphism and $H$ is the class of a point (these will be called conic bundles since $-K_X\cdot \pi^*(H)=2$),

    \item $\pi^*(H)$, where $\pi:X\to \mathbb{P}^2$ is a birational morphism and $H$ is the class of a line (these will be called twisted cubics since $-K_X\cdot \pi^*(H)=3$), 
    \item $\pi^*(2F+E)$, where $\pi:X\to \mathbb{F}_2$ is a birational morphism, $F$ is the class of a fiber and $E$ of the $(-2)$-curve,
 \item $\pi^*(-K_{Y}$), where $Y$ is a smooth surface 

 and $\pi:X\to Y$ is a birational morphism.
\end{itemize}
\end{proposition}

\begin{proof}
By \cite{ArLa}*{Proposition 1.1} the curves of negative intersection of $X$, i.e. $(-1)$- and $(-2)$-curves, generate the effective cone of $X$.
If $D\cdot E\geq 1$ for every $(-1)$-curve $E$, then $D+K_X$ is nef because it has non-negative intersection with every negative curve. 
It follows that $D=(D+K_X)+(-K_X)$, which implies that $D\sim-K_X$ since $D\in \BNef(X)$. 
If $D^2=0$, by the genus formula we have $2g-2=D\cdot K_X\leq 0$,
so either $-K_X\cdot D=0$ and then $D\sim-K_X$ by the Hodge index theorem, 
or $-K_X\cdot D=2$ and $D$ is a conic bundle. 
Thus we focus on those divisors $D$ with $D^2>0$ and such that there exist $(-1)$-curves orthogonal to $D$.
Contracting all such $(-1)$-curves we obtain that $D$ is the pull-back, by a birational morphism $\pi:X\to Y$ to smooth surface $Y$, 
of an element of $\BNef(Y)$ having positive intersection with all $(-1)$-curves.
If $\rho(Y)\geq 3$ then $D=\pi^*(-K_Y)$ by the previous argument. 
Otherwise $Y$ can be either $\mathbb P^2$, $\mathbb P^1\times\mathbb P^1$, $\mathbb P^2$ blown-up at one point or $\mathbb F_2$. 
Checking all elements of $\BNef(Y)$ for all of these, we get that $D$ is either the pull-back of a conic bundle (whose pullback is again a conic bundle), 
of the class of a line in $\mathbb{P}^2$, or of the class $2F+E$ in $\mathbb{F}_2$ as in the statement. 
\end{proof}

\begin{remark}
Observe that $-K_X$ when $X$ is a rational elliptic surface, conic bundles and twisted cubics are clearly elements of the Hilbert basis of the nef cone, 
because of their small self-intersection and the Hodge index theorem. 
Similarly one can prove the same for  $\pi^*(2F+E)$, since it 
has intersection $1$ with the class $\pi^*(F)$, which is a conic bundle.
In case $Y$ is a weak del Pezzo surface, $-K_Y$ can be either in the Hilbert basis of the nef cone or not, depending on the configuration of $(-2)$-curves of $Y$.  For example it can be easily proved that  $-K_Y$ is not in the Hilbert basis if $Y$ is a del Pezzo surface with $K_Y^2\geq 4$. On the other hand $-K_Y$ is in the Hilbert basis of the nef cone (in fact it generates an extremal ray) for any weak del Pezzo surface whose lattice of $(-2)$-curves has maximal rank.

\end{remark}

\subsection{Cox rings}\label{cox}

We now define our central object of study, which is defined in full generality in \cite{ADHL}*{Chapter 2}.
Given a smooth projective surface over $\mathbb{C}$ whose divisor class group $\Cl(X)$ 
is finitely generated and torsion free, its {\em Cox ring}  is the $\mathbb{C}$-algebra
\[
R(X)= \bigoplus_{D \in K} \Gamma(X,\mathcal O_X(D)),
\]
where $K$ is a subgroup of $\Div(X)$ that projects isomorphically onto $\Cl(X)$ via the quotient map $D\mapsto [D]$.
This is a $K$-graded algebra whose isomorphism class does not depend on the choice of $K$.
If $R(X)$ is a finitely generated algebra, then $X$ is called a {\em Mori dream space}.

Given a homogeneous element $f\in \Gamma(X,\mathcal O_X(D))$ of $R(X)$,
we will denote its {\em degree} with $\deg(f):=[D]\in \Cl(X)$ and 
we will say that it is a {\em defining section} 
of an effective divisor $E$ if  $E=\div(f)+D$.

We will say that a divisor class $[D]$ is a \emph{necessary degree} for $R(X)$
if any minimal set of homogeneous generators of $R(X)$ has an element of degree $[D]$. 

\begin{remark}\label{hilbeff}
It follows from the definition of $R(X)$ that the degrees of a set of homogeneous 
generators of $R(X)$ are a generating set of the effective monoid of $R(X)$.
This implies that the elements of the Hilbert basis of the effective cone of $X$ are necessary degrees.
In particular, this holds for all classes of integral curves with negative self-intersection.
\end{remark}

The following theorem \cite{ArLa} identifies which rational anticanonical surfaces with nef anticanonical class 
are Mori dream spaces.

\begin{theorem}
Let $X$ be a smooth projective rational surface such that $-K_X$ is nef. The Cox ring $R(X)$ is finitely generated 
if and only if $X$ is either
\begin{enumerate}
\item a weak del Pezzo surface, or
\item  a rational elliptic surface with finitely many $(-1)$-curves. 

\end{enumerate}
\end{theorem}

If $X$ is a relatively minimal rational elliptic surface, equivalent conditions for $X$ to be a Mori dream space are that 
the effective cone of $X$ is rational polyhedral, or that the lattice generated by the classes of $(-2)$-curves 
has rank $8$ \cite{ArLa}*{Theorem 4.2, Proposition 5.1}. 
These will be called {\em extremal rational elliptic surfaces}.
The types of the singular fibers of such elliptic surfaces have been classified in \cite{MiPe}.

We now recall some techniques for computing Cox rings, more precisely
for proving that certain effective degrees $[D]\in \Cl(X)$ are not necessary for $R(X)$.
The following two results are obtained considering Koszul exact sequences.

\begin{lemma}[\cite{B}*{Lemma I.5}]\label{koszul2}
Let $X$ be a smooth projective complex surface, 
$E_1,E_2$ be two disjoint curves of $X$ 
and $f_1,f_2\in R(X)$ be their defining sections.
Given $D\in \Div(X)$ there is an exact cohomology sequence:
\[
  \xymatrix@1@C16pt{
  H^0(X, D-E_1)\oplus H^0(X,D-E_2)\ar[r]^-{g^*} & 
 H^0(X,D)\ar[r] & H^1(X, D-E_1-E_2),
 }
\]
where $g^*(s,t)=sf_1+tf_2$. If $g^*$ is surjective, then $[D]$ is not a necessary 
degree for $R(X)$.
In particular, this holds if $h^1(X,D-E_1-E_2)=0$.
\end{lemma}

\begin{lemma}[\cite{ACL}] \label{koszul3}
Let $X$ be a smooth complex projective surface, $E_1,E_2,E_3$ be three curves 
of $X$ such that $E_1\cap E_2\cap E_3=\emptyset$ 
and $f_1, f_2, f_3\in R(X)$ be their defining sections. 
If $D\in\Div(X)$ is such that $h^1(X,D-E_i-E_j)=0$ for all distinct $i,j$ and $h^2(X,D-E_1-E_2-E_3)=0$, then there is a surjective map
$$\bigoplus_{i=1}^3 H^0(X,D-E_i)\to H^0(X,D),\quad (g_1,g_2,g_3)\mapsto f_1g_1+f_2g_2+f_3g_3.$$
In particular,  if this holds, $[D]$ is not a necessary 
degree for $R(X)$.
\end{lemma}

The following result will be also useful for computations.

\begin{lemma} \label{lemaT}
    Let $X$ be a smooth projective surface,  $D$ be an effective divisor 
    without components in its base locus and $A,B$ be two distinct curves such that $h^0(A,D_{|A})=1$. 
 If $h^0(X,D-B)>0$ and $A$ is not contained in the base locus of $|D-B|$, then $H^0(X,D)$ is generated by $H^0(X,D-A)s_A$ and $H^0(X,D-B)s_B$, where $s_A,s_B$ are defining 
sections of $A$ and $B$. 

\end{lemma}

\begin{proof}
Consider the following diagram:
 \begin{center}
     \begin{tikzcd}
            &                               & {H^0(X,D-B)} \arrow[d, "m_B"] &                                            &            \\
0 \arrow[r] & {H^0(X,D-A)} \arrow[r, "m_A"] & {H^0(X,D)} \arrow[r, "r_A"]   & {H^0(A,D_{|A})}\cong \mathbb{C},
\end{tikzcd}
 \end{center}
where $m_A$, $m_B$ are multiplications by defining sections of $A$ and $B$ respectively and $r_A$ is the restriction to $A$.

Since $A$ is not in the base locus of $|D-B|$ 
then $r_A\circ m_B$ is also surjective.
Given $f\in H^0(X,D)$, let $r_A(f)=c\in H^0(A,D_A)$ and let $g\in H^0(X,D-B)$ 
such that $r_A\circ m_B(g)=c$, 
then $f-m_B(g)\in H^0(X,D-A)$, proving the statement.
\end{proof}
 
\section{Main Theorem} 
In this section we will prove the following results.

\begin{theorem} \label{main1}
Let $X$ be a weak del Pezzo surface with $\rho(X)\geq 3$. The necessary degrees for ${R}(X)$ belong to the 
following list: 
\begin{enumerate}[(i)]
\item classes of $(-1)$-curves and of $(-2)$-curves,  
\item elements of $\BNef(X)$. 

\end{enumerate}
\end{theorem}

\begin{theorem} \label{main2}
Let $X$ be an extremal rational elliptic surface of index $m>1$.  
The necessary degrees for ${R}(X)$ belong to the 
following list:
\begin{enumerate}[(i)]
\item classes of $(-1)$-curves and of $(-2)$-curves, 
\item elements of $\BNef(X)$, 
 \item ample classes of the form $D\sim-\alpha K_X+E$, where $2\leq\alpha< m$ and $E$ is  a $(-1)$-curve.

\end{enumerate}

\end{theorem}

We prove several preliminary results first.

 \begin{lemma} \label{NiwdP}
    Let $X$ be a weak del Pezzo surface and $N_1$, $N_2$, $N_3$ be nonzero nef divisors on $X$. 
    Then, there exist nonzero nef divisors $N_1', N_2', N_3'$ on $X$ such that $N_1'+N_2'+N_3'\sim N_1+N_2+N_3$ and $N_1'\cap N_2'\cap N_3'=\emptyset$ unless $\rho(X)=9$ and $N_1+N_2+N_3\sim-3K_X$.
\end{lemma}

\begin{proof} If one of the three divisors is base point free, we can choose $N_i=N_i'$ for each $i$. 
Otherwise, by Proposition \ref{harbourne},
$\rho(X)=9$ and $N_1+N_2+N_3\sim-3K_X$.
\end{proof}

\begin{lemma}\label{Ni} 
  Let $X$ be a relatively minimal rational elliptic surface of index $m>1$ 
  
  and $N_1,N_2,N_3$ be nef nonzero divisors on $X$. 
  There exist nef nonzero divisors $N_1',N_2',N_3'$ on $X$ such that $N_1+N_2+N_3 \sim N_1'+N_2'+N_3'$ and $N_1'\cap N_2'\cap N_3'=\emptyset$ unless
    \begin{itemize}
 
 \item  $N_1+N_2+N_3\sim -3 K_X+E+E'$, where $E,E'$ are $(-1)$-curves such that $-K_X+E$ and $-K_X+E'$ have the same base point, or
    \item $N_1+N_2+N_3\sim -\alpha K_X+E$, where $\alpha\leq m$ and $E$ is a $(-1$)-curve, or
    \item $N_1+N_2+N_3\sim -\alpha K_X + D$, where $\alpha\leq m$ and $D$ is a nef and base point free divisor, or
    \item $N_1+N_2+N_3 \sim -\alpha K_X$, where $\alpha\leq m+1$.
    \end{itemize}
   
\end{lemma}

 \begin{proof}
  We will consider several cases according to the possible fixed loci of the divisors $N_i$
  (we will repeatedly use Proposition \ref{harbourne}). We will classify the distinct cases 
  according to the triple $(b_1,b_2,b_3)$, where $b_i$ is the dimension of the base locus of the divisor $N_i$.
Nefness can be easily checked looking at intersections with $(-1)$- and $(-2)$-curves.

\noindent {\em Case 1.} If $(b_1,b_2,b_3)\in \{(-1,-1,-1), (0,0,-1), (0,-1,-1), (1,-1,-1), (1,0,-1)\}$  we can take $N_i'\sim N_i$ for $i=1,2,3$.\\

\noindent {\em Case 2.} If $(b_1,b_2,b_3)=(0,0,0)$ , then 
         $$N_1\sim -\alpha_1K_X+E_1,\quad N_2\sim -\alpha_2K_X+E_2,\quad N_3\sim -\alpha_3K_X+E_3,$$
         where $E_i$ are $(-1)$-curves and $\alpha_i$ are positive integers. In this case, we reduce to Case 1 by taking
          $$N_1'\sim -(\alpha_1+\alpha_2+\alpha_3-2)K_X,\quad N_2'\sim -K_X+E_1+E_2,\quad N_3'\sim -K_X+E_3.$$
\noindent {\em Case 3.} If  $(b_1,b_2,b_3)=(1,0,0)$ we can assume 
$$N_1\sim -\alpha_1K_X,\quad N_2\sim -\alpha_2K_X+E_2,\quad N_3\sim -\alpha_3K_X+E_3$$ where $\alpha_i>0$ and $E_i$ are $(-1)$-curves.  
If $\Bs(N_2)\neq \Bs(N_3)$, then we can take $N_i\sim N_i'$.
Otherwise,  if $N_2\cap F=N_3\cap F$ we can replace $N_1, N_2$ with $N_1'=N_1-K_X$ and $N_2'=N_2+K_X$ as long as $\sum \alpha_i >3$.
Observe that $\Bs(N'_2)\neq \Bs(N_3)$ since $(-K_X)_{|F}$ is not trivial. The case $\alpha_1=\alpha_2=\alpha_3=1$ gives a first exception.

\noindent {\em Case 4.} If $(b_1,b_2,b_3)=(1,1,0)$, then $N_1\sim -\alpha_1 K_X$, $N_2\sim -\alpha_2 K_X$ and
 $N_3\sim -\alpha_3 K_X + E$, where $E$ is a $(-1)$-curve. In this case, we can take 
              $$N_1'\sim -(\alpha_1+\alpha_2+\alpha_3-m-1)K_X+ E,\quad N_2'\sim -K_X,\quad N_3'\sim -mK_X$$ 
              every time that $\alpha_1+\alpha_2+\alpha_3\geq m+1$. The case $\alpha_1+\alpha_2+\alpha_3\leq m$ gives a second exception.  
              
    \noindent {\em Case 5.} If $(b_1,b_2,b_3)=(1,1,-1)$ and $N_1\sim -\alpha_1 K_X$, $N_2\sim -\alpha_2 K_X$,
           then we obtain a third exception as long as $\alpha_1+ \alpha_2\leq m$. If $\alpha_1+ \alpha_2\geq m+1$, we can take 
              $$N_1'\sim -mK_X,\quad N_2'\sim -(\alpha_1+ \alpha_2-m)K_X,\quad N_3'\sim N_3.$$

\noindent {\em Case 6.} If $(b_1,b_2,b_3)=(1,1,1)$ then $N_i\sim -\alpha_i K_X$ for all $i$, 
which gives a last exception as long as $\alpha_1+\alpha_2+\alpha_3\leq m+1$. 
If $\alpha_1+\alpha_2+\alpha_3\geq m+2$, we can take
          $$N_1'\sim -mK_X,\quad N_2'\sim -(\alpha_1+ \alpha_2+\alpha_3-m-1)K_X,\quad N_3'\sim -K_X.$$
 \end{proof}

\begin{proposition}\label{na}
Let $X$ be either a weak del Pezzo surface or a rational elliptic surface of index $m>1$. 
 If $D$ is a nef divisor such that $[D]\not\in \BNef(X)$ and which is orthogonal to a smooth rational curve, then $[D]$ is not a necessary degree for $R(X)$.
\end{proposition}

\begin{proof}
Let $E$ be a smooth rational curve which is orthogonal to $D$. Consider the exact sequence
\[
\xymatrix{
0\ar[r] & H^0(X,D-E)\ar[r]& H^0(X,D)\ar[r]& H^0(E,D_{|E})=H^0(E,\mathcal O_E)\cong \mathbb C.
}
\]
By Proposition \ref{harbourne} nef divisors have no smooth rational curves in their fixed locus.
This implies that the sequence is also exact on the right. 
Since $[D]\not\in \BNef(X)$ then $D\sim D_1+D_2$, where $D_1,D_2$ are nef.
Thus $H^0(X,D)$ can be generated as a vector space by $f_EH^0(X,D-E)$ and $f_1f_2$, where $f_E$ is a defining section of $E$ 
and $f_i\in H^0(X,D_i)$ do not vanish along $E$. Thus in this case $[D]$ is not necessary to generate $R(X)$.
\end{proof}

\begin{corollary}\label{can}
Let $X$ be either a weak del Pezzo surface with $\rho(X)\geq 3$ 
 or an extremal rational elliptic surface of index $m>1$. 
Then $-\alpha K_X$ is not a necessary degree for $R(X)$ when $\alpha>1$.
\end{corollary}

\begin{proof}
If $X$ is either an elliptic surface or a weak del Pezzo surface containing $(-2)$-curves, 
this follows from Proposition \ref{na}. If $X$ is a del Pezzo surface this follows from \cite{ADHL}*{Theorem 5.2.2.1}.
\end{proof}

\begin{proposition} \label{alphaK}
     Let $X$ be an extremal rational elliptic surface of index $m>1$. 
      
     Let $D\sim -\alpha K_X+N$, where $\alpha$ is a positive integer and 
     $[N]\in\BNef(X)$. If $[D]$ is a necessary degree for $R(X)$ then $D\sim -\alpha K_X+E$ with $\alpha< m$, where $E$ is a $(-1)$-curve which intersects any $(-2)$-curve.

\end{proposition}

\begin{proof}
If $D$ is not ample, then we are done by Proposition \ref{na}.
We restrict to the case when it is ample, or equivalently when $N$ has positive intersection with any $(-2)$-curve. 
We consider several cases according to Proposition \ref{bajar}.

If $N$ is a conic bundle, since there are no $(-2)$-curves in its orthogonal, 
its reducible fibers are the union of two $(-1)$-curves $E,E'$ intersecting at one point,
thus $D\sim -\alpha K_X+E+E'$. We will apply Lemma \ref{koszul3} 
to shows that $[D]$ is not necessary in several cases.
If $\alpha\geq m+1$ Lemma \ref{koszul3}  can be applied with the divisors
\[
N_1=-K_X,\quad N_2=-K_X+E,\quad N_3=-mK_X.
\]
 Observe that $h^1(D-N_i-N_j)=0$ for all distinct $i,j$ by Proposition \ref{harbourne}, 
 and $h^2(D-N_1-N_2-N_3)=h^0(K_X-D+N_1+N_2+N_3)=0$ since the last divisor 
has negative intersection with $-K_X$. 
Moreover the divisors $N_i$ can be chosen to have empty intersection 
since $N_1$, $N_2$ intersect at one point and $N_3$ is base point free.

If $\alpha\leq m$ Lemma \ref{koszul3} can be applied with the divisors
\[
A_1=-(\alpha-1) K_X,\quad A_2=-K_X+E,\quad A_3=E',
\]
or the analogous divisors $A_1, A'_2, A'_3$ with the roles of $E$ and $E'$ exchanged, 
or the divisors
\[
B_1=-(\alpha-2) K_X,\quad B_2=-2K_X+E,\quad B_3=E'
\]
when $\alpha > 2$.
Observe that,  if the restrictions of $A_2$ and $A_3$ 
to $F$ are linearly equivalent, then the same is false
for $A'_2, A'_3$ unless $m=2$ and
for $B_2, B_3$ if $\alpha> 2$, since $\tau:=(-K_X)_{|F}$ is not trivial. 
In each case one can verify that  $h^1(D-N_i-N_j)=0$ for all distinct $i,j$ by Proposition \ref{harbourne} 
and $h^2(0)=0$.
This implies that the only exception is the case $\alpha=m=2$ 
and $D\sim -2K_X+E+E'$, where $E,E'$ are $(-1)$-curves 
intersecting in one point and $\tau+p\sim p'$, where 
$p=F\cap E$ and $p'=F\cap E'$. We consider this case later in the proof.

If $N\sim \pi^*(2F+E)$, where $\pi:X\to \mathbb F_2$ is a birational morphism, 
$F$ is the class of a fiber and $E$ is the class of the $(-2)$-curve of $\mathbb F_2$,
then $N\cdot \pi^*(E)=0$, contradicting our assumption on $N$.

Now assume $N=\pi^*(H)$, where $\pi:X\to \pp^2$ is a birational morphism 
and $H$ is the class of a line.   
Let $\pi=b\circ p$ where $b$ is the blow up over one point with exceptional divisor $E$.
We have that $\pi^*(H)=p^*(b^*(H))$, so we consider the three divisors
\[
N_1=p^*(b^*(H)-E)-(\alpha-1) K_X, \quad N_2=p^*(E),\quad N_3=-K_X.
\]
Notice that $N_1$ is nef and $-K_X\cdot N_1=2$, so it is base point free and $h^1(N_1)=0$ 
by Proposition \ref{harbourne}.
Moreover $h^1(N_2)=0$ because $N_2$ is a $(-1)$-curve and $h^1(-K_X)=0$.
Observe that we can take $N_1$, $N_2$, $N_3$ to have empty intersection up to 
linear equivalence since $N_2, N_3$ intersect at one point and 
$N_1$ is base point free. Thus $[D]$ is not necessary by Lemma \ref{koszul3}.

Now, assume that $N\sim \pi^*(-K_Y)$ where $\pi:X\to Y$ is a birational morphism onto a weak del Pezzo surface $Y$. 
Since $N$ has positive intersection with any $(-2)$-curve, then $N\sim -K_X+E_1+\dots +E_n$, where the $E_i$'s 
are disjoint $(-1)$-curves.
  
If $n>2$ consider the divisors 
$$
N_1=E_n,\quad N_2=-\alpha K_X+E_1+\cdots+E_{n-1},\quad N_3=-K_X.
$$ 
 Since $-K_X\cdot N_2\geq 2$, then $N_2$ is base point free by Proposition \ref{harbourne}.
Since $E_n$ and $-K_X$ intersect at one point, then $N_1, N_2, N_3$ can be chosen to have 
 empty intersection up to linear equivalence. Moreover $h^1(N_1)=h^1(N_2)=h^1(N_3)=h^2(0)=0$ 
  by the Riemann-Roch formula and Proposition \ref{harbourne}. Thus $[D]$ is not necessary by  Lemma \ref{koszul3}.
  
  If $n=1$, then $D\sim -(\alpha+1)K_X+E_1$. 
  If $\alpha+1=m$ we obtain that $[D]$ is not necessary by Lemma \ref{lemaT} with $A=F\in |-K_X|$ and $B=E$. 
  If $\alpha+1> m$ and $\alpha+1\not=0\, ({\rm mod}\ m)$, 
  then we can apply Lemma \ref{koszul2} with $N_1=-rK_X$ and $N_2=-mqK_X$,
  where $q,r$ are quotient and remainder of $\alpha+1$ mod $m$. 
  Observe that $N_1, N_2$ are linearly equivalent to disjoint curves, 
  since $-mqK_X$ is base point free. Moreover $h^1(D-N_1-N_2)=h^1(E_1)=0$.
  On the other hand, If $\alpha+1> m$ and $\alpha+1=0\, ({\rm mod}\ m)$, i.e. $D\sim -mqK_X+E_1$ with $q>1$, then
  we can apply Lemma \ref{koszul2} with $N_1=-m(q-1)K_X$ and $N_2=-mK_X$.
  If $\alpha+1\leq m$ we obtain an exception.

  If $n=2$, then $D\sim -(\alpha +1)K_X+E_1+E_2$ and we can apply the same arguments 
  of the case when $N$ is a conic bundle.  

  Finally, let us consider the case when $D\sim -2 K_X+E_1+E_2$ where $m=2$ and    $E_1,E_2$ are $(-1)$-curves such that $(E_1-E_2)_{|F}\sim (-K_X)_{|F}$ and $E_1+E_2$ intersects any $(-2)$-curve.
By the classification of singular fibers of extremal rational elliptic surfaces \cite{MiPe}, $X$ can only be of type $2A_3+2A_1$ or $4A_2$ (in every other case, $X$ has a fiber with at least 5 components, and then $E_1+E_2$ cannot intersect every $(-2)$-curve). Choosing a fiber with the highest number of components in each type, we have one of the following three cases (up to renaming):

\begin{center}
  \begin{tikzpicture}[scale=.4]
   \draw[thick,fill=black] (-1,1) circle (2 mm);
   \draw[thick,fill=black] (1,1) circle (2 mm);
   \draw[thick,fill=black] (-1,-1) circle (2 mm);
   \draw[thick,fill=black] (1,-1) circle (2 mm);
   \draw[thick,fill=gray] (0,2.5) circle (2 mm);
   \draw[thick,fill=gray] (0,-2.5) circle (2 mm);
      \draw (-1, 1) node[anchor=east]  {$G_1$};
         \draw (-1, -1) node[anchor=east]  {$G_2$};
            \draw (1, 1) node[anchor=west]  {$G_4$};
               \draw (1, -1) node[anchor=west]  {$G_3$};
                  \draw (0, 2.5) node[anchor=south]  {$E_1$};
                     \draw (0, -2.5) node[anchor=north]  {$E_2$};
    \draw (-1,1) -- (1,1);
    \draw (-1,1) -- (-1,-1);
    \draw (-1,-1) -- (1,-1);
    \draw (1,-1) -- (1,1);
 \draw (-1,1) -- (0,2.3);
 \draw (1,1) -- (0,2.3);
  \draw (-1,-1) -- (0,-2.3);
 \draw (1,-1) -- (0,-2.3);

\draw[thick,fill=gray] (8,-2.5) circle (2 mm);
\draw[thick,fill=gray] (8,2.5) circle (2 mm);
    \draw[thick,fill=black] (7,-1) circle (2 mm);
     \draw[thick,fill=black] (8,1) circle (2 mm);
   \draw[thick,fill=black] (9,-1) circle (2 mm);
\draw (7,-1) -- (8,1);
    \draw (8,1) -- (9,-1);
    \draw (7,-1) -- (9,-1);
        \draw (8,-2.3) -- (7,-1);
\draw (8,-2.3) -- (9,-1);
\draw (8,1) to[out=-20,in=-20] (8,2.3);
\draw (8,2.3) to[out=200,in=200] (8,1);

 \draw (8, 1) node[anchor=east]  {$G_1$};   
         \draw (7, -1) node[anchor=east]  {$G_2$};
            \draw (9, -1) node[anchor=west]  {$G_3$};
 \draw (8, 2.5) node[anchor=south]  {$E_1$};
  \draw (8, -2.5) node[anchor=north]  {$E_2$};

   \draw[thick,fill=gray] (16,-2.5) circle (2 mm);
     \draw[thick,fill=gray] (18,1) circle (2 mm);
    \draw[thick,fill=black] (15,-1) circle (2 mm);
      \draw[thick,fill=black] (16,1) circle (2 mm);
   \draw[thick,fill=black] (17,-1) circle (2 mm);
   \draw (16,1) -- (15,-1);
    \draw (15,-1) -- (17,-1);
    \draw (16,1) -- (17,-1);
 \draw (16, 1) node[anchor=east]  {$G_1$};   
         \draw (15, -1) node[anchor=east]  {$G_2$};
            \draw (17, -1) node[anchor=west]  {$G_3$};
 \draw (17,1) node[anchor=south]  {$E_1$};
  \draw (16,-2.5) node[anchor=north]  {$E_2$};
\draw (16,-2.3) -- (15,-1);
\draw (16,-2.3) -- (17,-1);

 \draw (17.8,1) -- (16, 1);
\draw (17.9,0.9) -- (17, -1);
  \end{tikzpicture}
\end{center}

In each of these cases, $G_2$ and $E_1$ are disjoint curves. We can then apply Lemma \ref{koszul2} to these curves: we have that $-2K_X+E_1+E_2 - G_2 - E_1$ is a generalized $(-1)$-curve and then $h^1(-2K_X+E_1+E_2 - G_2 - E_1)=0$. As such, $[D]$ is not a necessary degree for $R(X)$.
\end{proof}

\begin{lemma} \label{alphaKwdp}
     Let $X$ be a weak del Pezzo surface.
     If $D\sim -\alpha K_X+N$, where $\alpha$ is a positive integer and 
     $[N]\in\BNef(X)$, then $[D]$ is not a necessary degree.
\end{lemma}

\begin{proof}
The proof is similar to the one of the previous Proposition.

If $D\sim -\alpha K_X+E$ and $\alpha\geq 3$ we can  apply Lemma \ref{koszul3} 
with $N_1\sim -K_X$, $N_2=-(\alpha-1)K_X$, $N_3=E$. 
Moreover if $D\sim -\alpha K_X+E_1+E_2$, where $E_1,E_2$ are $(-1)$-curves, then 
we can  apply Lemma \ref{koszul3} 
with $N_1\sim -K_X$, $N_2=E_1$, $N_3=-(\alpha-1)K_X+E_2$. 

We now exclude the case when $\rho(X)=9$ and $D\sim -2 K_X+E$, where $E$ is a $(-1)$-curve passing through the base point of $|-K_X|$ and has positive intersection with any $(-2)$-curve. Since $D\cdot E=1$ and $h^1(-2K_X)=0$ by Proposition \ref{harbourne} there is an exact sequence
\[
\xymatrix{
0\ar[r]& H^0(X,-2K_X)\ar[r]& H^0(X,D)\ar[r]^{r\quad} & H^0(E,\mathcal O_E(1))\ar[r]& 0.
}
\]
In particular the vector space $H^0(X,-2K_X)f_E$, where $f_E$ is a generator of $H^0(X,E)$,
has codimension $2$ in $H^0(X,D)$.
Let $f\in H^0(X,-K_X+E)$ not vanishing along $E$. Consider the composite map
\[
\xymatrix{
H^0(X,-K_X)\ar[r]^{\quad m_f} &H^0(D)\ar[r]^{r\quad }& H^0(E,\mathcal O_E(1)),
}
\]
where $m_f$ is the multiplication by $f$ and $r$ the restriction map to $E$.
The image of this map is one, since the image of the restriction map $H^0(X,-K_X)\to H^0(E,{-K_X}_{|E})$ is one.
This shows that $H^0(X,D)$ can be generated as a vector space by $H^0(X,D-E)f_E$, $H^0(X,-K_X)f$ 
and an element not vanishing at the base point of $|-K_X|$.

Since $-K_X\cdot E=1$ and $E$ passes through the base point of $|-K_X|$, then 
there is a curve $R$ such that $E+R\sim -K_X$. 
Observe that $R$ is a generalized $(-2)$-curve since $-K_X\cdot R=0$.
Since $D$ is ample, then $E$ intersects any $(-2)$-curve of $X$.
This implies that all the $(-2)$-curves of $X$ are components of $R$ 
and that $R$ is either irreducible or the union of two $(-2)$-curves intersecting at one point.

Let $E'$ be a $(-1)$-curve disjoint from $E$ and consider the divisor $D'=D-E'$.
We will show that $D-E'$ is nef. 
If $C$ is a $(-2)$-curve, then $D'\cdot C\geq 0$ since $D\cdot C=E\cdot C>0$ 
and $E'\cdot C\leq 1$ since $C$ is a component of an element in $|-K_X|$.
If $C$ is a $(-1)$-curve then $D'\cdot C=(-2K_X+E-E')\cdot C=2+E\cdot C-E'\cdot C$. 
Since the intersection of two $(-1)$-curves of $X$ is at most $3$, then  $D'\cdot C\geq 0$
unless $E\cdot C=0$ and $E'\cdot C=3$.
In this case, after contracting $E$, the images of $E'$ and $C'$ would be two $(-1)$-curves 
on a del Pezzo surface of degree two with intersection number $3$, which is not possible.

By Proposition \ref{harbourne} $D-E'$ is base point free.
Thus a non-zero element in $H^0(X,D-E')f_{E'}$, where  $f_{E'}$ is a generator of $H^0(X,E')$, is an element 
of $H^0(X,D)$ not vanishing at the base point of $|-K_X|$.
\end{proof}

\begin{corollary}\label{s2}
Let $X$ be either a weak del Pezzo surface with $\rho(X)\geq 3$ 
 or an extremal rational elliptic surface of index $m>1$. 
 Let $D=D_1+D_2$, with $D_1,D_2\in \BNef(X)$, and such that $D$ is not linearly equivalent 
 to the sum of more than two nef divisors. 
Then $[D]$ is not a necessary degree for $R(X)$. 
\end{corollary}

\begin{proof}
By Proposition \ref{na} we can assume $D$ to be ample.
Then $D\sim -K_X+(D+K_X)$, where $N:=D+K_X$ is nef and has positive intersection with 
any $(-2)$-curve. Moreover $N\in \BNef(X)$. 
We then have, by Proposition \ref{alphaK} and Lemma \ref{alphaKwdp}, that $[D]$ is not a necessary degree. \end{proof}

 \begin{proof}[Proof Theorem \ref{main1}]
The degrees of integral curves with negative self-intersection are always necessary to generate $R(X)$, see Remark \ref{hilbeff}.
Moreover, if $D$ is not nef and is not linearly equivalent to one such curve, then $|D|$ has a curve $C$ in its base locus, 
thus $[D]$ is not necessary to generate $R(X)$ since $H^0(X,D)$ can be generated by $H^0(X,D-C)$ and a defining element of $C$.
Thus we will assume $D$ to be nef.

The case when $[D]$ is the sum of two elements in the Hilbert basis of the nef cone, and not more,
is considered in Corollary \ref{s2}.

We now assume that $[D]$ is the sum of at least three elements in the Hilbert basis of the nef cone. 
In this case there are $N_1,N_2,N_3$ non-zero nef divisors such that $D-\sum_{i=1}^3 N_i$ is nef. 

Observe that $h^1(D-N_i-N_j)=0$ for all distinct $i,j$ 
and $h^2(D-N_1-N_2-N_3)=0$  by Proposition \ref{harbourne} or Kawamata-Viehweg vanishing theorem, 
since $D-N_i-N_j$ and $D-N_1-N_2-N_3$ are nef.
Moreover, by Lemma \ref{NiwdP} we can choose $N_1, N_2, N_3$ to have empty intersection 
unless $\rho(X)=9$ and $D\sim -\alpha K_X$ for some integer $\alpha\geq 3$.
We conclude that $[D]$ is not necessary to generate $R(X)$ by Lemma \ref{koszul3} and  Corollary \ref{can}.
\end{proof}
 
\begin{remark}
    Theorem \ref{main1} gives a general statement and an alternative proof for the results about Cox rings of weak del Pezzo surfaces in \cites{DD, DDD}
    and \cite{ADHL}*{Theorem 5.2.2.3}, and extends them to Picard number $9$.
\end{remark}

\begin{proof}[Proof of Theorem \ref{main2}]

By the same arguments in the proof of Theorem \ref{main1} we can assume $D$ 
to be nef.  The case when $[D]$ is the sum of two elements  in the Hilbert basis of the nef cone, and not more,
is considered in Corollary \ref{s2}. We now assume that $[D]$ is the sum of at least three elements in the Hilbert basis of the nef cone. 
In this case there are $N_1,N_2,N_3$ non-zero nef divisors such that $D-\sum_{i=1}^3 N_i$ is nef. 
By Lemma \ref{Ni} we can choose $N_1, N_2, N_3$ to have empty intersection 
unless $D$ is of one of the following types 

\begin{enumerate}

   \item $D\sim -3 K_X+E+E'$, where $E,E'$ are $(-1)$-curves such that $-K_X+E$ and $-K_X+E'$ have a base point in common,
  \item $D\sim-\alpha K_X+E$, where $3\leq\alpha\leq m$ and $E$ is a $(-1)$-curve,
    \item $D\sim-\alpha K_X+N$, where $2\leq\alpha\leq m$ and $N\in\BNef(X)$ is base point free, 
    \item $D\sim-\alpha K_X$, where $\alpha\geq 3$,

\end{enumerate}  
Cases (i)-(iii) are considered in Lemma \ref{alphaK} and its proof,
and case (iv) in Corollary \ref{can}. 
 
In  what follows, we assume that $D$ is not of these types.  
 
We have that $h^2(D-\sum_{i=1}^3 N_i)=h^0(K_X-D+\sum_{i=1}^3 N_i)=0$ by Serre duality and 
the fact that  $-K_X+D-\sum_{i=1}^3 N_i$ is a non-zero effective divisor.
Let $A_{ij}:=D-N_i-N_j$. If $-K_X\cdot A_{ij}>0$ for all distinct $i,j$, then $h^1(A_{ij})=0$ by Proposition \ref{harbourne}.
Thus  $[D]$ is not necessary to generate $R(X)$ by Proposition \ref{koszul3} with the divisors $N_1, N_2, N_3$.

Assume now that  $h^1(A_{ij})\not=0$ for some  distinct $i,j$. 
Then $A_{ij}\sim-\alpha K_X$ for some positive integer $\alpha\geq m$ by Proposition \ref{harbourne} and 
the Hodge index theorem.
Thus $D\sim -\alpha K_X+N_i+N_j$. 

We can assume that neither $N_i$ nor $N_j$ is linearly equivalent to $-K_X$ 
since this case is considered in Lemma \ref{alphaK}.
Consider the nef divisors 
$$A_1=N_i,\quad A_2=N_j-(\alpha-1)K_X,\quad A_3=-K_X.$$ 
Note that $h^2(D-A_1-A_2-A_3)=h^2(0)=0$,  $h^1(D-A_1-A_2)=h^1(-K_X)=0$ 
and $h^1(D-A_2-A_3)=h^1(N_i)=0$ by Proposition \ref{harbourne}.
Moreover,  $h^1(D-A_1-A_3)=h^1(N_j-(\alpha-1)K_X)=0$, otherwise $N_j-(\alpha -1)K_X  \sim -sK_X$
for some positive integer $s$ by Proposition \ref{harbourne}, so that $N_j\sim -K_X$. 
Thus $[D]$ is not necessary to generate $R(X)$ by Proposition \ref{koszul3} with the divisors $A_1, A_2, A_3$.
 
 To conclude, we are left with the exceptions in the statement of Lemma \ref{alphaK}.
\end{proof}

Observe that the class of type (iii) in Theorem \ref{main2} does not appear when $E$
has zero intersection with some $(-2)$-curve. This happens when the index is sufficiently small compared to the number 
of components of a reducible fiber. For example we obtain the following result when $m=2,3$.

\begin{corollary} \label{2halphenspecified}
 Let $X$ be an extremal rational elliptic surface of index $m$. If either $m=2$, or
   $m=3$ and the configuration of singular fibers of $X$ is not of type $4A_2$, 
 then the necessary 
 degrees for $R(X)$ are either classes of negative curves or elements of $\BNef(X)$. 
\end{corollary}

\begin{proof} It follows from Theorem \ref{main2} and the classification of singular fibers of extremal rational elliptic surfaces \cite{MiPe}.
For $X$ of index 2, in all configurations the fibration has a fiber with at least $3$ components, thus there 
is a component which does not intersect $E$. For $X$ of index 3, in all but the configuration in the statement the fibration has a fiber with at least $4$ components, thus there is a component which does not intersect $E$.
\end{proof}

 \begin{theorem}\label{necesarios} 
Let $X$ be a smooth projective rational surface such that $-K_X$ is nef.
Then the following hold.
\begin{enumerate}
\item If $X$ is an elliptic surface, $-K_X$ is a necessary degree for $R(X)$ if and only if $m=1$ and the elliptic fibration has a unique reducible fiber, or $m>1$ and $F\in |-K_X|$ is irreducible.
\item A conic bundle is a necessary degree for $R(X)$ if and only if the associated morphism $\pi:X\to \mathbb P^1$ has a unique reducible fiber.
\item A twisted cubic is a necessary degree for $R(X)$ if and only if it the associated morphism $\pi:X\to \mathbb P^2$
contracts all negative curves to one point.
\item A class of type $\pi^*(2F+E)$, where $\pi:X\to \mathbb F_2$ is a birational morphism, is not a necessary degree for $R(X)$.
\item A class of type $\pi^*(-K_Y)$, where $\pi:X\to  Y$ is a birational morphism distinct from the identity, to a weak del Pezzo surface $Y$ and $|-K_X|$ contains irreducible elements, is a necessary degree for $R(X)$ if and only if $m>1$ and $Y$ is a del Pezzo surface of degree one.
\end{enumerate}
\end{theorem}
 
\begin{proof}
The first three items are due to the fact that the class considered is an element of the Hilbert basis 
of the nef cone. 
The first item follows from the fact that if $X$ is an extremal elliptic surface, then the only way
to write $-K_X$ as a sum of effective divisors is as a sum of $(-2)$-curves.
Thus $-K_X$ is a necessary degree of $R(X)$ if and only if  $H^0(-K_X)$ can not be generated by products of defining 
sections of $(-2)$-curves. 
The second item is similar.
In the third item, $H^0(\pi^*(H))$ can be generated by  defining sections of reducible curves if and only if it blows-up at least 
two distinct points in $\mathbb P^2$.
In the fourth item, a class of type $\pi^*(2F+E)$ is not necessary by Lemma \ref{lemaT} applied to the divisors $B=\pi^*(E)$ 
and $A$ any $(-1)$-curve in the exceptional divisor of $\pi$. Finally, for the fifth item, we can write $\pi^*(-K_Y)=-K_X+E_1+\cdots+E_n$, where the $E_i$'s are generalized $(-1)$-curves with $E_i\cdot E_j=0$ for $i\neq j$. If $Y$ contains a $(-2)$-curve $G$, then a class of type $\pi^*(-K_Y)$ is not necessary by Lemma \ref{lemaT} applied to the divisors $A=\pi^*(G)$ and $B=E_1$, so $Y$ contains no $(-2)$-curves. We conclude by \cite{BP}*{Theorem 3.2}, which states that $-K_Y$ is necessary to generate $R(Y)$ if and only if $Y$ is a del Pezzo surface of degree one. For the converse of this item, 
observe that if $Y$ is a del Pezzo surface of degree one, then 
$\pi^*(-K_Y)=-K_X+E$, where $E$ denotes the exceptional divisor over a point $p$, 
has a unique reducible fiber since $-K_Y$ has no reducible fibers and there is a unique element of $|-K_Y|$ that passes through $p$ (since $m>1$). This implies that $\pi^*(-K_Y)=-K_X+E$ is a necessary degree. 
\end{proof}

\begin{remark}
    The hypothesis on the irreducibility of $F\in|-K_X|$ is only needed in the fifth item (we use the fact that $G$ is not in the base locus of $|-K_X|$ to apply Lemma \ref{lemaT}). If we do not assume this and $\pi$ blows up at least two different points, writing $\pi^*(-K_Y)=-K_X+F_1\cdots+F_n$, where $F_i, F_j$ are sums of exceptional curves which are disjoint   for $i\neq j$, we can apply Lemma \ref{lemaT} to the divisors $B=F_1$ and $A=F_2$. This implies that a class of type $\pi^*(-K_Y)$ can be necessary only if $\pi$ is a blow-up over one point.
\end{remark}

\begin{remark}
    If $X$ is a weak del Pezzo surface, we have partial results regarding the necessity of $-K_X$ to generate $R(X)$. 
    For instance: if $X$ is del Pezzo, $-K_X$ is necessary for $R(X)$ if and only if $X$ is of degree one, by \cite{BP}*{Theorem 3.2}; and it is never necessary if $X$ is toric.
    Observe that by Proposition \ref{na} $-K_X$ is necessary only if it belongs to the Hilbert basis of the nef cone. If $X$ is \textit{extremal}, that is the lattice generated by classes of $(-2)$-curves has maximum rank, then by \cite{Y}*{Lemma 3.2} there is an extremal rational elliptic surface that dominates $X$. By the same arguments in  Example \ref{Example  wdp}, one can prove that if $X$ is not toric and is dominated by a rational elliptic surface $Y$ whose Cox ring has exactly one relation, then $-K_X$ is necessary for $R(X)$ if and only if $-K_Y$ is necessary for $R(Y)$.
\end{remark}

\section{Examples}

Given a nef anticanonical surface with finitely generated Cox ring,
Theorem \ref{main1} and Theorem \ref{main2} provide a finite list $L$ of candidates for the 
degrees of a minimal homogeneous set of generators of $R(X)$.
Observe that classes of curves with negative self-intersection are always necessary by Remark \ref{hilbeff}. Moreover Theorem \ref{necesarios} implies that certain nef classes are necessary.

In order to find an irredundant set of degrees in specific examples, 
we further analyze the nef elements of $L$ by applying some tests based on Koszul Lemmas \ref{koszul2} and \ref{koszul3}. 
This is achieved with the help of MAGMA \cite{Magma} programs, inspired by those in \cite{ACL}*{Section 3.3},
which are contained in two different libraries available \href{https://github.com/excelsofiasco/cox-rings-anticanonical-surfaces}{here}. 

In this section, we give some examples of computation of Cox rings of nef anticanonical surfaces.

\subsection{Halphen rational elliptic surfaces of type \texorpdfstring{$E_8$}{E8}} \label{Example  2-Halphen}
Consider the following pencils of cubics in $\mathbb{P}^2$:
$$(x_1^3+x_0^2x_2)+tx_2^3=0$$
$$(x_1^3+x_0^2x_2+x_1^2x_2)+tx_2^3=0.$$
Blowing up the nine infinitely near base points of the pencil, we get a rational elliptic surface $\pi:X\to \mathbb P^1$ with $m=1$ 
whose negative curves have the following intersection graph

\begin{center}
  \begin{tikzpicture}[scale=.4]
    \draw (-1,1) node[anchor=east]  {};
    \foreach \x in {1,...,7}
    \draw[thick,xshift=\x cm,fill=black] (\x cm,0) circle (2 mm);
    \foreach \y in {0,...,6}
    \draw[thick,xshift=\y cm] (\y cm,0) ++(.3 cm, 0) -- +(14 mm,0);
    \draw[thick,xshift=0 cm] (-2 cm,0) circle (2 mm);
    \draw[thick,xshift=0 cm,fill=black] (0 cm,0) circle (2 mm);
    \draw[thick,fill=gray] (10 cm,2 cm) circle (2 mm);
    \draw[thick] (10 cm, 3mm) -- +(0, 1.4 cm);
    \draw[thick] +(-3 mm, 0mm ) -- +(-1.7 cm, 0 cm);
  \end{tikzpicture}
\end{center}
where the black dots correspond to the $(-2)$-curves coming from the exceptional divisors 
of the first $8$ blow-ups over the point $p=(1,0,0)$, the gray dot corresponds to the $(-2)$-curve 
coming from a line passing through $p$ with the same direction as the second blowing up,
and the white dot corresponds to the $(-1)$-curve coming from the final blow-up over $p$. 
The curves corresponding to black and gray dots form a singular fiber of $\pi$, of type $II^*$. These are the surfaces called $X_{22}$ and $X_{211}$ in the notation of \cites{MiPe, AGL}. 
The Cox rings of these surfaces, degrees of generators and relations, are described in \cite{AGL}*{Table 5}. Observe that the matrix of degrees of the generators of $R(X)$ is the same for the two surfaces. 
We now consider rational elliptic surfaces with a fiber of type $II^*$ and $m>1$. 

We now take a $2$-Halphen transform of either $X_{22}$ or $X_{211}$ \cite{ADHL}*{Proposition 5.1.4.6}: 
by contracting the $(-1)$-curve given by the last blow-up over $p$ and blowing up over a 
point $q$ in a smooth fiber of $\pi$ such that $q-p$ is of $2$-torsion (in the group structure of the elliptic curve), 
we get a new surface $X'$. 
The intersection graph of negative curves of  $X'$ is the following

\begin{center}
  \begin{tikzpicture}[scale=.4]
    \draw (-1,1) node[anchor=east]  {};
    \foreach \x in {1,...,7}
    \draw[thick,xshift=\x cm,fill=black] (\x cm,0) circle (2 mm);
     \foreach \x in {1,...,7}
    \draw[xshift=\x cm] (\x cm,-0.2 cm) node[anchor=north]  {$G_{\x}$};
    \foreach \y in {0,...,6}
    \draw[thick,xshift=\y cm] (\y cm,0) ++(.3 cm, 0) -- +(14 mm,0);
    \draw[thick,xshift=0 cm,fill=white] (0 cm,0) circle (2 mm);
    \draw (0 cm, -0.2 cm) node[anchor=north]  {$E_1$};
    \draw[thick,fill=gray] (10 cm,2 cm) circle (2 mm);
    \draw (10.2 cm, 2 cm) node[anchor=west]  {$L$};
    \draw[thick] (10 cm, 3mm) -- +(0, 1.4 cm);

     \draw[thick,fill=black] (2 cm,2 cm) circle (2 mm);
     \draw (2.2 cm, 2 cm) node[anchor=west]  {$T$};
     \draw[thick] (2 cm, 3mm) -- +(0, 1.4 cm);
     \draw[thick] (14 cm,4 cm) circle (2 mm);
     \draw (14.2 cm, 2 cm) node[anchor=west]  {$L_{pq}$};
     \draw (14.2 cm, 4 cm) node[anchor=west]  {$E_q$};
     \draw[thick] (14 cm,2 cm) circle (2 mm);
      \draw[thick] (14 cm, 3mm) -- +(0, 1.4 cm);
   \draw[thick] (14 cm, 23mm) -- +(0, 1.4 cm);
  \end{tikzpicture}
\end{center}
where again the black dots represent $(-2)$-curves ($G_1,\dots, G_7$ 
are those coming from the original exceptional divisors, and $T$ is a new one), 
the white dots are $(-1)$-curves ($E_q$ is the exceptional divisor over $q$,  $L_{pq}$ is the proper transform of the line through $p$ and $q$, 
and $E_1$ is the last exceptional divisor over $p$),
 and the gray dot is the original special line through $p$ (we call it $L$). 
 We know that $X'$ only has three $(-1)$-curves by \cite[]{LT}, 
 so we have them all. By writing all these in terms of the basis $H,E_1,\dots,E_8,E_q$ of $\Pic(X')$, 
 where $H$ is the class of a general line and the $E_i$'s are the classes of the 
 exceptional divisors over $p$ from last to first, we get that
$$G_i=E_{i+1}-E_i,\qquad L=H-E_8-E_7-E_6,\qquad L_{pq}=H-E_8-E_q.$$
Since that the class of a fiber is 
both $-2K_{X'}\sim 6H-\sum_{i}^8 2E_i$ and $T+2G_1+3G_2+4G_3+5G_4+6G_5+4G_6+2G_7+4L$ 
by Kodaira's classification of singular fibers, we get that
$$T\sim3H-E_2-E_3-E_4-E_5-E_6-E_7-E_8-2E_q.$$

By applying the Magma tests described before, we obtain that a necessary degree for $R(X)$ is either the class of a negative curve 
or one the following classes in the Hilbert basis of the nef cone:
\[
C_1= 5H-E_1-E_2-E_3-E_4-E_5-E_6-E_7-E_8-4E_q, 
\]
\[
C_2=-K_{X'},\quad C_3= H - E_q.
\]
Observe that $C_1^2=1$ and the negative curves orthogonal  to it form a connected diagram, 
thus it is a twisted cubic whose associate morphism contracts all the negative curves to a single point. 
The class $C_3$ is a conic bundle with a unique reducible fiber.
Thus $C_1,C_2,C_3$ are necessary degrees by Remark \ref{necesarios}.

\begin{remark}
    Notice that all twisted cubics of $X'$ are
    \begin{itemize}
        \item $5H-E_1-E_2-E_3-E_4-E_5-E_6-E_7-E_8-4E_q$: gives a model of $X'$ as a blow-up in one point,
        \item $4H-E_3-E_4-E_5-E_6-E_7-E_8-3E_q$: gives a model of $X'$ as a blow-up in two points, three times above one and six above the other,
        \item $H$: gives a model of $X'$ as a blow-up in two points, one time above one and eight above the other,
        \item $2H-E_7-E_8-E_q$: gives a model of $X'$ as a blow-up in two points, two times above one and seven above the other,
        \item $3H-E_5-E_6-E_7-E_8-2E_q$: gives a model of $X'$ as a blow-up in two points, four times above one and five above the other.
    \end{itemize}
    These models realize all possible examples of $2$-Halphen pencils that give 
    an $\tilde{E_8}$ type fiber, as proven by Zanardini in \cite{Z}*{Theorem 5.15}.
     \end{remark}

 In  \cite{HZ}*{Example 4.8} Hattori and Zanardini provide the following explicit example of a 3-Halphen pencil of type $E_8$.  Consider the cubic $C$ given by $z^2y + x(y^2 + xz) = 0$, the conic $Q$ given by $y^2 + xz = 0$ and the line $L$ given by $y = 0$. Then the pencil generated by $2Q+5L$ and $3C$ is a $3$-Halphen pencil. We can calculate classes of divisors using \cite{Z}*{Proposition 4.4} and \cite{Z}*{Lemma 4.5}: these imply that we have blow ups over two points $p_1,p_2$. Moreover, the multiplicity of $Q$ and $L$ in the pencil tells us that we have the following (dual) configuration of curves:

\begin{center}
   \begin{tikzpicture}[scale=.4]
\foreach \x in {1,...,8}
    \draw[thick,xshift=\x cm,fill=black] (\x cm,0) circle (2 mm);
    \draw[xshift=1 cm] (1 cm,0.2cm ) node[anchor=south]  {$2Q$};
    \draw[xshift=2 cm] (2 cm,0.2 cm) node[anchor=south]  {$4S_3$};
    \draw[xshift=3 cm] (3 cm,0.2 cm) node[anchor=south]  {$6S_2$};
    \draw[xshift=4 cm] (4 cm,0.2 cm) node[anchor=south]  {$5L$};
    \draw[xshift=5 cm] (5 cm,0.2 cm) node[anchor=south]  {$4R_1$};
    \draw[xshift=6 cm] (6 cm,0.2 cm) node[anchor=south]  {$3R_2$};
    \draw[xshift=7 cm] (7 cm,0.2 cm) node[anchor=south]  {$2R_3$};
     \draw[xshift=8 cm] (8 cm,0.2 cm) node[anchor=south]  {$R_4$};

    \foreach \y in {1,...,7}
    \draw[thick,xshift=\y cm] (\y cm,0) ++(.3 cm, 0) -- +(14 mm,0);

\draw[thick,fill=black] (6 cm, -2 cm) circle (2 mm);
   
    \draw[thick] (6 cm, 0 cm ) -- (6 cm,-2 cm);

\draw (6 cm, -2.5 cm) node[anchor=north]  {$3S_1$};

\end{tikzpicture}
\end{center}
We call, from last one to first one, $E_1,E_2,E_3,E_4,E_5$ the exceptional divisors over $p_1$, $F_1,F_2,F_3,F_4$ the exceptional divisors over $p_2$, and use $Q,L_1$ for the classes of the proper transforms of the curves $Q,L$. By using the basis of $\Pic(X)$ given by $\{H,E_i,F_i\}$, we get that the $(-2)$-curves of $X$ are given by
$$R_i=E_{i+1}-E_i,\quad S_i=F_{i+1}-Fi$$
$$L_1=H-E_5-F_3-F_4,\quad Q=2H-E_1-E_2-E_3-E_4-E_5-F_4.$$
Moreover we can find all the $(-1)$-curves of $X$. We find that the nef necessary degrees are in the list:
\[
C_1=-K_X,\quad C_2=H-F_4, \quad C_3=2H-F_1-F_2-F_3-F_4, 
\]
\[
C_4=7H-3E_2-3E_3-3E_4-3E_5-F_1-2F_2-2F_3-2F_4.
\]
All of them are necessary by Remark \ref{necesarios}. 

In \cite{LT} Laface and Testa give a method to compute the classes of curves with negative self-intersection of a family of $m$-Halphen rational surfaces for any $m>0$. Roughly speaking, they show that the ways one can embed the lattice $\Lambda$ generated by classes of $(-2)$-curves of an $m$-Halphen rational elliptic surface in the $E_8$ lattice are parametrized by the group $\Ext(E_8/\Lambda,C_m)$ \cite{LT}*{Lemma 2.5} and they give a way to, given $(m,\xi)$ with $\xi \in \Ext(E_8/\Lambda,C_m)$, construct an $m$-Halphen rational elliptic surface with associated lattice $\Lambda$ \cite{LT}*{Proposition 2.6}. This essentially allows us to skip obtaining a geometric model for the surface.

By using a Magma program, available \href{https://github.com/excelsofiasco/cox-rings-anticanonical-surfaces/blob/main/Laface%20and%20Testa's%20program}{here},  
we compute all classes of negative curves of a $4$-Halphen and a $5$-Halphen rational surface $X$ of type $E_8$. 
This allows us to find the necessary degrees for their Cox rings.
We resume our results in Table \ref{halphenE8}.

 \begin{table}[h]    
 
         \begin{tabular}{c|ccccc}\small
        $m$  &  Necessary degrees & Negative curves & Conic bundles & Models to $\mathbb{P}^2$ & $-K_X$ \\ \midrule
       $1$  &   $13$ & 10 & 1 & 1& 1\\
       $2$  &   $15$ & 12 & 1 & 1& 1\\
        $3$  &   $18$ & 14 & 3 & 0& 1\\
        $4$  &   $23$ & 19 & 2 & 1& 1\\
        $5$  &   $31$ & 24 & 5 & 1& 1\\[2pt]
    \end{tabular}
    
    \caption{Halphen surfaces of type $E_8$}\label{halphenE8}
    \end{table}

\begin{remark} 
Let $T_1,\dots,T_n$ be the generators of $R(X)$.
Given a birational model $\pi: X\to \mathbb P^2$, consider those generators $T_1,\dots, T_r$ 
whose associated curves map to curves in $\mathbb P^2$ 
and let $f_1,\dots,f_r$ be their equations in $\mathbb C[x,y,z]$. 
By \cite{AGL}*{Corollary 4.2} the ideal of relations $I(X)$ of $R(X)$ is given by 
the saturation $I$ of the kernel of the homomorphism
\[
\mathbb C[T_1,\dots,T_r]\to \mathbb C[x,y,z],\ T_i\mapsto f_i
\]
with respect to the variables $ T_{r+1},\dots, T_n$, if $\dim(I)=\dim(I(X))$.
If we choose $H$ as a model for the $2$-Halphen surface $X'$ considered above, the equation of the plane curves are:
$$  x^5 + x^4y + z^3y^2 + x^3yz - xy^3z - y^4z,\quad x+y, \quad x,$$
    $$ z^2x + y^2z - x^3 - x^2z, \quad y^2z-x^3-x^2z, \quad z. $$
Unfortunately, for computational limits, we are not able to compute the ideal $I$.
\end{remark}

\subsection{A 2-Halphen rational elliptic surface of type \texorpdfstring{$D_8$}{D8}} \label{Example  2-Halphen D8}

We consider the model in \cite{Z}*{Example 7.30}. Let $C$ be a smooth plane cubic, $p_1$ a flex point of $C$ and $L_4$ its inflection line. Let $L_1$ be a line
through $p_1$ which is tangent to $C$ at another point $p_2$, and let $L_3$ be another line through $p_1$ which is tangent to $C$ at a different point $p_3$. Let $L_2$ be the line joining $p_2$ and $p_3$, and $p_4$ the third intersection point of $L_2$ and $C$. The pencil generated by $2L_1+2L_2+L_3+L_4$ and $2C$ is a $2$-Halphen pencil, and we have to blow up three times over $p_1$, $p_2$, two times over $p_3$ and once over $p_4$ to get a $2$-Halphen rational surface of type $D_8$. In fact, we get the following (dual) configuration of curves:
\begin{center}
   \begin{tikzpicture}[scale=.4]
\foreach \x in {1,...,5}
    \draw[thick,xshift=\x cm,fill=black] (\x cm,0) circle (2 mm);
    \draw[xshift=1 cm] (0.5 cm,0 cm) node[anchor=east]  {$2L_2$};
    \draw[xshift=2 cm] (2 cm,-0.2 cm) node[anchor=north]  {$2S_1$};
    \draw[xshift=3 cm] (3 cm,-0.2 cm) node[anchor=north]  {$2S_2$};
    \draw[xshift=4 cm] (4 cm,-0.2 cm) node[anchor=north]  {$2L_1$};
    \draw[xshift=5 cm] (5.5 cm, 0 cm) node[anchor=west]  {$2R_1$};
    \foreach \y in {1,...,4}
    \draw[thick,xshift=\y cm] (\y cm,0) ++(.3 cm, 0) -- +(14 mm,0);
\draw[thick,fill=black] (12 cm, 2) circle (2 mm);
    \draw[thick,fill=black] (12 cm, -2) circle (2 mm);
    \draw[thick,fill=black] (0 cm, 2) circle (2 mm);
    \draw[thick,fill=black] (0 cm, -2) circle (2 mm);
\draw[thick] (10 cm,0) -- (12 cm,2);
\draw[thick] (10 cm,0) -- (12 cm,-2);
\draw[thick] (2 cm, 0 ) -- (0 cm,2);
    \draw[thick] (2 cm, 0 ) -- (0 cm,-2);

\draw (-0.5 cm, 2 cm) node[anchor=east]  {$2G_1$};
\draw (-0.5 cm, -2 cm) node[anchor=east]  {$L_4$};
\draw (12.5 cm, 2 cm) node[anchor=west]  {$T_1$};
\draw (12.5 cm, -2 cm) node[anchor=west]  {$L_3$};
\end{tikzpicture}
\end{center}

We call, from last one to first one, $E_1,E_2,E_3$ the exceptional divisors over $p_1$, $F_1,F_2,F_3$ the exceptional divisors over $p_2$, $G_1,G_2$ the exceptional divisors over $p_3$ and $J$ the exceptional divisor over $p_4$. By using the basis of $\Pic(X)$ given by $\{H,E_i,F_j,G_k,J\}$, we get that the $(-2)$-curves of $X$ are given by
$$R_i=E_{i+1}-E_i,\qquad S_i=F_{i+1}-F_i,\qquad T_1=G_2-G_1,$$
$$L_1=H-E_3-F_3-F_2,\qquad L_2=H-F_3-G_2-J,$$
$$L_3=H-E_3-G_2-G_1,\qquad L_2=H-E_3-E_2-E_1.$$
Moreover we can find the classes of all the $(-1)$-curves of $X$ using a Magma program that can be found \href{https://github.com/excelsofiasco/cox-rings-anticanonical-surfaces/blob/main/Program%20to%20find%20(-1)-curves}{here}. In total, this surface has 18 $(-1)$-curves.

We obtain that the necessary degrees for $R(X)$ are either classes of negative curves or one of the following elements in the Hilbert basis of the nef cone:
\[
C_1=-K_X,\quad C_2=4H-2E_2-2E_3-F_1-F_2-F_3-2G_2-J,
\]
\[
C_3=2H-F_1-F_2-F_3-J.
 \]
 The classes $C_2$ and $C_3$ are conic bundles with only one reducible fiber. Thus $C_1,C_2,C_3$ are necessary
by Remark \ref{necesarios}.

\subsection{A 2-Halphen rational elliptic surface of type \texorpdfstring{$E_7+A_1$}{E7+A1}} \label{Example  2-Halphen E7+A1}

We consider the model given in \cite{Z}*{Example 7.45}. 
Let $D$ be a cubic with a node at $p_1$. We choose a flex point $p_2$, denote the respective inflection line by $L_1$, 
and choose a line $L_2$ through $p_2$ which intersects $D$ at two other points, say $p_3$ and $p_4$. 
Finally, we construct a cubic $C$ through $p_1,\dots,p_4$ such that $C$ intersects $D$ with multiplicity $5$ at $p_2$ 
and intersects $L_1$ with multiplicity $3$ at $p_2$. 
The pencil generated by $D + 2L_1 + L_2$ and $2C$ is a $2$-Halphen pencil.
Blowing up six times over $p_2$  and once over $p_1,p_3,p_4$ we get a $2$-Halphen rational surface of type $E_7+A_1$. 
In fact, we get the following configuration of curves:

\begin{center}
   \begin{tikzpicture}[scale=.4]
\foreach \x in {1,...,7}
    \draw[thick,xshift=\x cm,fill=black] (\x cm,0) circle (2 mm);
    \draw[xshift=1 cm] (1 cm,0.2cm ) node[anchor=south]  {$D$};
    \draw[xshift=2 cm] (2 cm,0.2 cm) node[anchor=south]  {$2R_5$};
    \draw[xshift=3 cm] (3 cm,0.2 cm) node[anchor=south]  {$3R_4$};
    \draw[xshift=4 cm] (4 cm,0.2 cm) node[anchor=south]  {$4R_3$};
    \draw[xshift=5 cm] (5 cm,0.2 cm) node[anchor=south]  {$3R_2$};
    \draw[xshift=6 cm] (6 cm,0.2 cm) node[anchor=south]  {$2R_1$};
    \draw[xshift=7 cm] (7 cm,0.2 cm) node[anchor=south]  {$L_2$};

    \foreach \y in {1,...,6}
    \draw[thick,xshift=\y cm] (\y cm,0) ++(.3 cm, 0) -- +(14 mm,0);

\draw[thick,fill=black] (8 cm, -2 cm) circle (2 mm);
   
    \draw[thick] (8 cm, 0 cm ) -- (8 cm,-2 cm);

\draw (8 cm, -2.5 cm) node[anchor=north]  {$2L_1$};

\end{tikzpicture}
\end{center}

We call, from last one to first one, $E_1,E_2,E_3,E_4,E_5,E_6$ the exceptional divisors over $p_2$, $F_1$ the exceptional divisor over $p_3$, $G_1$ the exceptional divisor over $p_4$ and $H_1$ the exceptional divisor over $p_1$. 

The classes of the $(-2)$-curves, using the basis of $\Pic(X)$ given by $\{H,E_i,F_1,G_1,H_1\}$, are:
$$R_i=E_{i+1}-E_i,\quad L_1:= H-E_6-E_5-E_4,\quad L_2:=H-E_6-F_1-G_1,$$
$$D=3H-E_2-E_3-E_4-E_5-E_6-F_1-G_1-H_1,$$
$$S_1=3H-E_1-E_2-E_3-E_4-E_5-E_6-2G_1-H_1,$$
$$S_2=3H-E_1-E_2-E_3-E_4-E_5-E_6-2F_1-H_1.$$
Thus we can find all the $(-1)$-curves on $X$ with the Magma program that can be found \href{https://github.com/excelsofiasco/cox-rings-anticanonical-surfaces/blob/main/Program%20to%20find%20(-1)-curves}{here}. In total, this surface has $18$ $(-1)$-curves. 
We obtain that the only nef necessary degree of $R(X)$ is $-K_X$.

\subsection{A 2-Halphen rational elliptic surface of type \texorpdfstring{$A_8$}{A8}} \label{Example  2-Halphen A8}

We consider the model given in \cite{Z}*{Example 7.13}. 
Let $L_1,\dots, L_6$ be  six lines intersecting as in the picture below: 
\begin{center}
    \includegraphics[scale=0.3]{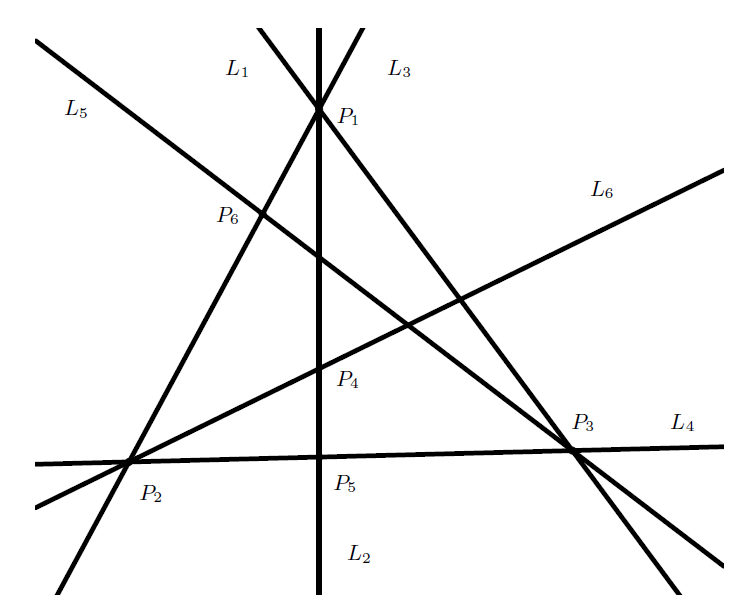}
\end{center}
and such that we can choose a smooth cubic $C$ passing through $p_4, p_5, p_6$ that is tangent with multiplicity 2 to $L_1$ at $p_1$, to $L_5$ at $p_3$ and to $L_6$ at $p_2$. The pencil generated by $L_1+\cdots+L_6$ and $2C$ is a $2$-Halphen pencil, and we have to blow up two times over $p_1,p_2,p_3$ and once over $p_4,p_5,p_6$ to get a $2$-Halphen rational surface of type $A_8$. In fact, we get the following (dual) configuration of curves:
\begin{center}
    \begin{tikzpicture}[scale=.25]
\node[draw=none,minimum size=3cm,regular polygon,regular polygon sides=9] (a) {};
\foreach \x in {1,2,...,9}
  \draw[thick,fill=black] (a.corner \x) circle (2 mm);
\draw[thick] (a.corner 1) -- (a.corner 2) ;
\draw[thick] (a.corner 2) -- (a.corner 3) ;
\draw[thick] (a.corner 3) -- (a.corner 4) ;
\draw[thick] (a.corner 4) -- (a.corner 5) ;
\draw[thick] (a.corner 5) -- (a.corner 6) ;
\draw[thick] (a.corner 6) -- (a.corner 7) ;
\draw[thick] (a.corner 7) -- (a.corner 8) ;
\draw[thick] (a.corner 8) -- (a.corner 9) ;
\draw[thick] (a.corner 9) -- (a.corner 1) ;
\draw (a.corner 1) node[anchor=south]  {$T_1$};
\draw (a.corner 2) node[anchor=south]  {$L_1$};
\draw (a.corner 3) node[anchor=east]  {$L_6$};
\draw (a.corner 4) node[anchor=east]  {$L_5$};
\draw (a.corner 5) node[anchor=north]  {$L_2$};
\draw (a.corner 6) node[anchor=north]  {$R_1$};
\draw (a.corner 7) node[anchor=west]  {$L_3$};
\draw (a.corner 8) node[anchor=west]  {$S_1$};
\draw (a.corner 9) node[anchor=south]  {$L_4$};

\end{tikzpicture}
\end{center}

We call, from last one to first one, $E_1,E_2$ the exceptional divisors over $p_1$, $F_1,F_2$ the exceptional divisors over $p_2$, $G_1,G_2$ the exceptional divisors over $p_3$, $H_1$ the exceptional divisor over $p_4$, $I_1$ the exceptional divisor over $p_5$ and $J_1$ the exceptional divisor over $p_6$. By using the basis of $\Pic(X)$ given by $\{H,E_i,F_j,G_k,H_1,I_1,J_1\}$, we get that the $(-2)$-curves of $X$ are given by
$$R_1=E_2-E_1,\quad S_1=F_2-F_1,\quad T_1=G_2-G_1,$$
$$L_1=H-E_1-E_2-G_2,\quad L_2=H-I_1-H_1-E_2,$$
$$L_3=H-E_2-J_1-F_2,\quad L_4=H-G_2-F_2-I_1.$$
$$L_5=H-J_1-G_1-G_2,\quad L_6=H-F_1-F_2-H_1.$$
Thus we can find the classes of all the $(-1)$-curves of $X$ using the Magma program that can be found \href{https://github.com/excelsofiasco/cox-rings-anticanonical-surfaces/blob/main/Program%20to%20find%20(-1)-curves}{here}. In total, this surface has $24$ $(-1)$-curves.
We get that the nef necessary degrees of $R(X)$ belong to the following list (in the basis above):
\[
\begin{array}{l}
C_1=-K_X,\\
C_2=19H-6E_1-6E_2-2F_1-8F_2-7G_1-7G_2-9H_1-4I_1-5J_1,\\
C_3=9H-3E_1-3E_2-2F_1-2F_2-2G_1-3G_2-5H_1-I_1-4J_1,\\
C_4=10H-3E_1-3E_2-F_1-4F_2-4G_1-4G_2-5H_1-2I_1-2J_1,\\
C_5=11H-4E_1-4E_2-2F_1-2F_2-3G_1-3G_2-H_1-6I_1-5J_1,\\
C_6=7H-E_1-E_2-2F_1-2F_2-G_1-2G_2-3H_1-3I_1-4J_1,\\
C_7=6H-2E_1-2E_2-3F_1-3F_2-2G_1-2G_2-I_1-J_1,\\
C_8=13H-2E_1-2E_2-4F_1-4F_2-3G_1-3G_2-5H_1-6I_1-7J_1,\\
C_9=9H-E_1-4E_2-3F_1-3F_2-4G_1-4G_2-3H_1-2I_1-J_1,\\
C_{10}=17H-2E_1-8E_2-6F_1-6F_2-7G_1-7G_2-5H_1-4I_1-3J_1,\\
C_{11}=5H-2E_1-2E_2-F_1-F_2-G_1-G_2-3I_1-2J_1,\\
C_{12}=13H-4E_1-4E_2-6F_1-6F_2-5G_1-5G_2-H_1-2I_1-3J_1,\\
C_{13}=17H-6E_1-6E_2-4F_1-4F_2-5G_1-5G_2-9H_1-2I_1-7J_1.
\end{array}
\]
Notice that  $C_3,C_4,C_6,C_7,C_9, C_{11}$ are conic bundles with only one reducible fiber. The other classes, except for $C_1$, are twisted cubics contracting all negative curves to a single point. All of them  are necessary by Remark \ref{necesarios}. 

\subsection{A weak del Pezzo surface of Picard number 7} \label{Example  rank 7}

Let us consider the $E_6$ cubic surface 
$$S=\{(x,y,z,w) : xy^2+yw^2+z^3=0\}\subseteq \mathbb{P}^3$$ that was studied by Hassett and Tschinkel in \cite{HT}*{Section 3} and is described by Derenthal in \cite{DD}*{Section 5.6} (in the latter, it is called a type xx surface). As Hassett and Tschinkel show, the Cox ring of $S$ has generators in degree $E_1,\dots,E_7,A_1,A_2,A_3$ where $E_1,\dots,E_6$ are $(-2)$-curves, $E_7$ is a $(-1)$-curve and the extended Dynkin diagram of negative curves of $S$ is 

\begin{center}
    \begin{tikzpicture}

    \draw (0,0) node {$A_1$};
    \draw (-1,1) node {$A_3$};
    \draw (-1,-1) node {$A_2$};
    \draw (0,1) node {$E_7$};
    \draw (1,1) node {$E_4$};
    \draw (2,1) node {$E_5$};
    \draw (1,0) node {$E_1$};
    \draw (2,0) node {$E_3$};
    \draw (3,0) node {$E_6$};
    \draw (2,-1) node {$E_2$};
 \draw (-0.2,-0.2) -- (-0.8,-0.8);
  \draw (-0.2,0.2) -- (-0.8,0.8);
   \draw (0.3,0) -- (0.7,0);
   \draw (1.3,0) -- (1.7,0);
   \draw (2.3,0) -- (2.7,0);
   \draw (-0.7,-1) -- (1.7,-1);
   \draw (-0.7,1) -- (-0.3,1);
    \draw (0.3,1) -- (0.7,1);
   \draw (1.3,1) -- (1.7,1);
     \draw (2.8,-0.2) -- (2.2,-0.8);
  \draw (2.8,0.2) -- (2.2,0.8);
  \draw (-1,0.7) -- (-1,-0.7);
  \draw (0,1) circle (3 mm);
\end{tikzpicture}
\end{center}

By considering the basis of $\Pic(S)$ given by $E_1,\dots,E_7$, we have 
\[
A_3=(2,3,4,4,5,6,3)=-K_S,\ A_1=(0,1,1,2,2,2,2), \ A_2=(1,1,2,3,3,3,3).
\]
Let us write these divisors in a more manageable form: considering the basis of $\Pic(S)$ given by $H,T_1,\dots,T_6$, where $H$ is a general line and $$
    T_1=E_7+E_4,\quad T_2=E_7+E_4+E_5, \quad T_3=E_7+E_4+E_5+E_6 $$ $$
    T_4=E_7+E_4+E_5+E_6+E_3, \quad T_5=E_7+E_4+E_5+E_6+E_3+E_1,\quad T_6=E_7
$$
(which are all the $(-1)$-curves excepting $E_7+E_4+E_5+E_6+E_2$) we have that
$A_1= H-T_5$ and $A_2= H$,
and so the nef necessary degrees of the Cox ring of $S$ are a conic bundle $A_1$, a twisted cubic $A_2$ and the anticanonical divisor $A_3$. 

\subsection{Weak del Pezzo surfaces dominated by extremal jacobian elliptic surfaces}\label{Example  wdp}

Let $X$ be a weak del Pezzo surface such that there is a birational morphism $\pi: Z\to X$, where $Z$ is an extremal rational elliptic surface with a section. By \cite{Y}*{Theorem 3.4} this holds for all weak del Pezzo surfaces with $K_X^2\leq 7$ which are extremal i.e. when the lattice generated by classes of $(-2)$-curves has maximal rank. 
Let $I$ be the ideal of $R(Z)$ generated by the elements of the form $f-1$, where $f$ is the defining element of an integral component of the exceptional divisor of $\pi$ and $1$ has degree $0$.
By the criterion described in \cite{AGL}*{\S 5} 
the Cox ring of $X$ is given by the quotient $R':=R(Z)/I$ if $\dim(R')=\dim(R(X))$ and ${\rm Spec}(R')$ is irreducible. 
Since the Cox rings of many examples of extremal jacobian elliptic surfaces have been computed in \cite{AGL}, this allows one to compute the Cox ring of the weak del Pezzo surfaces dominated by them.

For example, let $Z$ be a jacobian rational elliptic surface with singular fibers of type $II^*II$  
(this is the surface $X_{22}$  in the notation of \cite{AGL}*{Table 1}).
Looking at the matrix of necessary degrees of ${R}(X_{22})$ given in \cite{AGL}*{Table 4} we notice that the class of the unique $(-1)$-curve is in the last column. Thus, setting the last variable equal to $1$, we get a ring in $12$ variables with the following degree matrix

$$\left(\begin{array}{cccc|cccccccc}
1  & 1  & 1  & 3  & 0  & 0  & 0  & 0  & 0  & 0  & 0  & 0  \\
-1  & -1  & 0  & -1  & 1  & 0  & 0  & 0  & 0  & 0  & 0  & 0  \\
-1  & 0  & 0  & -1  & -1  & 1  & 0  & 0  & 0  & 0  & 0  & 0  \\
-1  & 0  & 0  & -1  & 0  & -1  & 1  & 0  & 0  & 0  & 0  & 0  \\
0  & 0  & 0  & -1  & 0  & 0  & -1  & 1  & 0  & 0  & 0  & 0  \\
0  & 0  & 0  & -1  & 0  & 0  & 0  & -1  & 1  & 0  & 0  & 0  \\
0  & 0  & 0  & -1  & 0  & 0  & 0  & 0  & -1  & 1  & 0  & 0  \\
0  & 0  & 0  & -1  & 0  & 0  & 0  & 0  & 0  & -1  & 1  & 0  \\
0  & 0  & 0  & -1  & 0  & 0  & 0  & 0  & 0  & 0  & -1  & 1  \\
\end{array}\right).$$
and with a unique relation: 
$$ T_1T^2_3 + T^3_2 T^2_5 T_6 - T_4T_8T^2_9T^3_{10}T^4_{11}T^5_{12}.$$
Similarly, if $Z$ is the jacobian rational elliptic surface with singular fibers of type  $II^*2I_1$   (this is the surface $X_{211}$ in the notation of \cite{AGL}*{Table 1}), we get that the necessary degrees of ${R}(X)$ are

    $$\left(\begin{array}{cccc|cccccccc}
1  & 1  & 1  & 3  & 0  & 0  & 0  & 0  & 0  & 0  & 0  & 0  \\
-1  & -1  & 0  & -1  & 1  & 0  & 0  & 0  & 0  & 0  & 0  & 0  \\
-1  & 0  & 0  & -1  & -1  & 1  & 0  & 0  & 0  & 0  & 0  & 0  \\
-1  & 0  & 0  & -1  & 0  & -1  & 1  & 0  & 0  & 0  & 0  & 0  \\
0  & 0  & 0  & -1  & 0  & 0  & -1  & 1  & 0  & 0  & 0  & 0  \\
0  & 0  & 0  & -1  & 0  & 0  & 0  & -1  & 1  & 0  & 0  & 0  \\
0  & 0  & 0  & -1  & 0  & 0  & 0  & 0  & -1  & 1  & 0  & 0  \\
0  & 0  & 0  & -1  & 0  & 0  & 0  & 0  & 0  & -1  & 1  & 0  \\
0  & 0  & 0  & -1  & 0  & 0  & 0  & 0  & 0  & 0  & -1  & 1  \\
\end{array}\right)$$

and the ideal of relations is generated by
$$T_1T^2_3 + T^3_2 T^2_5 T_6 - T_4T_8T^2_9T^3_{10}T^4_{11}T^5_{12} + T_1T^2_2T^2_5T^2_6T^2_7T^2_8T^2_9T^2_{10}T^2_{11}T^2_{12}.$$

\end{document}